\documentclass[letterpaper,10pt]{amsart}
\usepackage{amsmath}
\usepackage{calc}
\usepackage[dvipsnames]{xcolor}
\usepackage{accents}
\usepackage[all]{xy}                      %
  
\CompileMatrices                            % Faster

\UseTips                                    % Use

\input xypic
\usepackage[bookmarks=true]{hyperref}       % Hyperref
\usepackage{amssymb,latexsym,amsmath,amscd}
\usepackage{xspace}
\usepackage{color}
\usepackage{graphicx}
\usepackage{dsfont}
\reversemarginpar

\theoremstyle{plain}
\newtheorem{theorem}{Theorem}[section]
\newtheorem*{theorem*}{Theorem}

\newtheorem{corollary}[theorem]{Corollary}

\theoremstyle{definition}

\newtheorem{remark}[theorem]{Remark}

\newcommand{\enm}[1]{\ensuremath{#1}}          %
\newcommand{\cal}[1]{\mathcal{#1}}

\newcommand{\CC}{\enm{\mathbb{C}}}

\newcommand{\II}{\enm{\mathbb{I}}}

\newcommand{\RR}{\enm{\mathbb{R}}}

\newcommand{\PP}{\enm{\mathbb{P}}}

\newcommand{\KK}{\enm{\mathbb{K}}}

\newcommand{\Oo}{\enm{\cal{O}}}

\newcommand{\Uu}{\enm{\cal{U}}}

\renewcommand{\phi}{\varphi}
\renewcommand{\theta}{\vartheta}
\renewcommand{\epsilon}{\varepsilon}

\renewcommand{\to}[1][]{\xrightarrow{\ #1\ }}

\newcommand{\old}[1]{}

\date{}

\begin{document}

\title[Hadamard product]
{Projective surfaces not as Hadamard products and the dimensions of the Hadamard joins}
\author{Edoardo Ballico}
\address{Dept. of Mathematics\\
 University of Trento\\
38123 Povo (TN), Italy\\
https://orcid.org/0000-0002-1432-7413}
\email{edoardo.ballico@unitn.it}
\thanks{The author is a member of GNSAGA of INdAM (Italy).}

\begin{abstract}
We study the dimensions of Hadamard products of $k\ge 3$ varieties if we allow to modify $k-1$ of them by the action of a general projective linear transformation.
We also prove that the join of a variety not contained in a coordinate hyperplane with a ``nice'' curve always has the expected dimension.
\end{abstract}

\maketitle

\section{Introduction}
Let $\PP^n$ be a $n$-dimensional complex projective space. We fix a system of homogeneous system $x_0,\dots ,x_n$ of coordinates of $\PP^n$. Fix $P=[p_0:\dots :q_n]\in \PP^n$,
$Q= [q_0:\dots :q_n]\in \PP^n$ such that $p_iq_i\ne 0$ for at least one $i\in \{0,\dots ,n\}$, then the Hadamard product $P\star Q$ of $P$ and $Q$ is the point $P\star Q= [p_0q_0,\dots ,p_nq_n]\in \PP^n$. If $p_iq_i=0$ for all $i$, then $P\star Q$ is not defined or at least it is a symbol, $\infty$ or $\emptyset$, not a point of $\PP^n$.
Let $X$ and $Y$  be irreducible subvarieties of $\PP^n$ such that there are at least one $P\in X$ and $Q\in Y$ with $P\star Q$  well-defined. With this assumption
$P_1\star Q_1$ is well-defined for a non-empty Zariski open subset $U$ of $X\times Y$. The Hadamard product $X\star Y$ is the closure in $\PP^n$ of the union of all $P_1\star Q_1$ with $(P_1,Q_1)\in U$.
The Hadamard product  $X\star Y$ is the irreducible variety of dimension at most $\min \{n,\dim X+\dim Y\}$ and an important question is giving conditions on $X$ and $Y$ such that
$\dim X\star Y =\min \{n,\dim X+\dim Y\}$. In the same way or using the associative law and induction we prove the Hadamard product of $X_1\star \cdots \star X_k$ of $k\ge 3$ varieties assuming the existence of $P_j= [p_{0j}:\dots :p_{nj}]\in X_j$, $1\le j\le k$, and $i\in \{0,\dots ,n\}$ such that $\prod_{j=1}^{k} p_{ij}\ne 0$, i.e. such that $P_1\star \cdots \star P_k$ is well-defined.

The Hadamard product  of $2$ points of $\PP^n$ depends on the fixed coordinate space and all related notions are not invariants for the automorphism group $PGL(n+1)$ of $\PP^n$. Hence the Hadamard products is not invariant for change of coordinates. For instance there are linear spaces $V_1,V_2$ such that $V_1\star V_2 =\emptyset$, while $\dim g(V_1)\star g(V_2)$ has dimension $\dim V_1+\dim V_2$ for almost all $g\in PGL(V)$.
This is the main difficulty on this topic which was born for strict applied reasons (\cite{CMS,CTY}) and continue with  several different tools, often related to Commutative Algebra, Projective Algebraic Geometry  and Tropical Geometry (which is an important tool) (\cite{MS15}, \cite[\S 2.5]{BC24}). Recently, C. Bocci and E. Carlini published a book on these topics containing a description of the main strands and a long bibliography (\cite{BC24}). One of the result of this note is a negative answer to a question raised in this book (\cite[Research Question 12.1.5]{BC24}) (see Theorem \ref{i4}). 

Our main results are an attempt to being less vague and less wasteful when in algebro-geometric jargon one says ``take general varieties'' (among a fixed class of varieties). With a genericity assumption several strong results on the joins $X_1\star \cdots \star X_k$ of varieties are known, but often they require $n$ very large, of the order of the product of the dimensions $\dim X_i$, $1\le i\le k$ (\cite[\S 6.1]{BC24}). 

To state our results we need the following standard notation. Set $H_i:= \{x_i=0\}$, $i=0,\dots ,n$. Set $\Delta_{n-1}:= H_0\cup \cdots \cup H_n$. Note that if $P\in (\PP^n\setminus \Delta_{n-1})$, then $P\star Q$ is well-defined for all $Q\in \PP^n$. Let $PGL(n+1)$ be the quotient of the linear group $GL(n+1)$ by the non-zero multiples of the identity matrix. We have $\mathrm{Aut}(\PP^n) =PGL(n+1)$.

Fix varieties $X, Y\subset \PP^n$. Instead of asking that the pair $(X,Y)$ is general in a prescribed class of pairs of projective varieties, we fix one such pair $(X,Y)$, even very special,
and we compute $\dim g(X)\star Y$ for a general $P\in PGL(n+1)$.
We prove the following results.

\begin{theorem}\label{i2}
Fix integral varieties $X_i\subset \PP^n$, $1\le i\le k$, $k\ge 2$, such that $X_k\nsubseteq \Delta_{n-1}$. Take a general $(g_1,\dots ,g_{k-1})\in PGL(n+1)^{k-1}$.
Then $$\dim g_1(X_1)\star\cdots \star g_{k-1}(X_{k-1})\star X_k =\min \{n,\dim X_1+\cdots +\dim X_k\}.$$\end{theorem}
Note that the assumption $X_k\nsubseteq \Delta_{n-1}$ is essential, because if $X_k\subseteq H_i$, then $W\star X_k\subseteq H_i$ for all integral varieties $W$. 

As an immediate corollary of Theorem \ref{i2} we get the following result.

\begin{theorem}\label{i3}
Fix integral varieties $X_i\subset \PP^n$, $1\le i\le k$, $k\ge 2$. Take a general $(g_1,\dots ,g_k)\in PGL(n+1)^k$.
Then $$\dim g_1(X_1)\star\cdots \star g_k(X_k) =\min \{n,\dim X_1+\cdots +\dim X_k\}.$$\end{theorem}

The obvious inspiration for statements similar to Theorems \ref{i2} and \ref{i3} came to us from the notions of generic initial ideals and universal Gr\"{o}bner bases (\cite{AL94}, \cite[\S 15.9]{eis}).

In the next statement ``very general'' means ``outside countably many proper Zariski closed subsets of the $(\binom{d+3}{3}-1)$-dimensional projective space parametrizing the set of all degree $d$ surfaces of $\PP^3$''.

\begin{theorem}\label{i4}
Let $W\subset \PP^3$ be a very general surface   of degree $d\ge 4$. There are no curves $X,Y\subset \PP^3$ such that $W =X\star Y$.
\end{theorem}

``Very general ''is more restrictive than ``general'' and it is the only place in which we use that $\CC$ is uncountable. All the other results of this paper are true without any modification of their proof for an arbitrary algebraically closed field of characteristic $0$. Recall that the parameter space of the set of all  degree $d$ surfaces of $\PP^3$ is a projective space $|\Oo_{\PP^3}(d)|$ of dimension $\binom{d+3}{3}-1$. Very general implies that it is Zariski dense in $|\Oo_{\PP^3}(d)|$. Thus to prove Theorem \ref{i4} it is sufficient to prove the existence of a pair $(\Uu_d, \phi)$, where $\Uu_d$ is a quasi-projective variety, $\phi : \Uu_d\to |\Oo_{\PP^3}(d)|$ is a morphism, $\dim \Uu_d\le \binom{d+3}{3}-2$ and any degree $d$ surface of the form $X\star Y$ for some curves $X, Y$ is contained in $\mathrm{Im}(\phi)$. Being very general also means that its complement is a countable union of subvarieties of dimension at most $ \binom{d+3}{3}-2$. Thus for any probability measure on $|\Oo_{\PP^3}(d)|$ locally equivalent to the Lebesgue measure this complement has measure $0$. Thus with high probability a random degree $d\ge 4$ surface of $\PP^3$ is not a star product of $2$ curves.

Concerning the joins with a curve we prove the following results (see \cite[Cor. 1.4 and Remark 1.6]{a} for the classical case of joins and secant varieties).

\begin{theorem}\label{i1}
Let $X\subset \PP^n$ be an integral variety such that $X\nsubseteq \Delta_{n-1}$. Let $Y\subset \PP^n$ be an integral curve such that $Y\cap \Delta_{n-1} \ne \emptyset$
and $Y$ is not contained in a binomial hypersurface in the sense of \cite[Ch. 8]{BC24}. Then $\dim X\star Y = \min \{n,\dim X+1\}$.
\end{theorem}

\begin{corollary}\label{i01}
Let $Y\subset  \PP^n$ be a curve such that $Y\nsubseteq \Delta_{n-1}$ and $Y$ is not contained in a binomial hypersurface. Then all Hadamard secant products of $Y$ have the expected dimension.
\end{corollary}

\section{The proofs}

\begin{proof}[Proof of Theorem \ref{i2}:]
First assume $k=2$. Set $m:= \dim X_1$. Fix a general $(P,Q) \in X_1\times X_2$ and any $g\in PGL(n+1)$ such that $g(P)\notin \Delta_{n-1}$. Since $Q$ is general in $X_2$ and $X_2\nsubseteq \Delta_{n-1}$, $Q\notin \Delta_{n-1}$. Thus $g(P)\star Q$ is well-defined and $g(P)\star Q\notin \Delta_{n-1}$. The point $g(P)\star Q$ is a general point of $g(X_1)\star X_2$. It is sufficient to prove that for a general $g\in PGL(n+1)$ the tangent space of $g(X_1)\star X_2$ at $g(P)\star Q$ has dimension $\min \{n,\dim X+1\}$. For any $g$ such $g(P)\notin \Delta_{n-1}$
the tangent space of $g(X_1)\star X_2$ at $g(P)\star Q$ is the two linear spaces spanned by  $T_{g(P)}g(X_1)\star Q$ and $P\star T_QX_2$, two linear spaces passing through $g(P)\star Q$ and of dimension $\dim X_1$ and $X_2$, respectively (\cite[Lemma 1.6]{BC24}).  

Let $G(m+1,n+1)$ denote the Grassmannian of all $m$-dimensional linear subspaces of $\PP^n$. Let $G(m+1,n+1;P)$ (resp. $G(m+1,n+1;P\star Q)$)
denote the Grassmannian of all $m$-dimensional linear subspaces of $\PP^n$ containing $P$ (resp. $P\star Q$). $G(m+1,n+1;P)$ is an irreducible projective variety of dimension $mn$ isomorphic to the variety $G(m,n)$.  Since $Q\notin \Delta_{n-1}$, the map
$V\mapsto V\star Q$ induces an isomorphism between $G(m+1,n+1;P)$ and $G(m+1,n+1;P\star Q)$. 
Let $G_P$ (resp. $G_{P\star Q}$) denote the set of all $g\in PGL(n+1)$ such that $g(P) =P$ (resp. $g(P\star Q) =g(P\star Q)$).
Since  $Q\notin \Delta_{n-1}$, $G_P$ and $G_{P\star Q}$ are isomorphic algebraic groups. The group $G_P$ (resp. $G_{P\star Q}$) is irreducible and it acts transitively
on the variety $G(m+1,n+1;P)$ (resp. $G(m+1,n+1;P\star Q)$). Note that $g(T_PX_1)=T_{g(P)}g(X_1)$ for all $g\in G_P$. The transitivity of $G_P$ mean that for a general $V_1\in G(m+1,n+1;P)$ there is $g\in G_P$ such that $g(T_pX_1)=V_1$. Note that $V_1\star Q$ is a general element of $G(m+1,n+1;P\star Q)$. Since $P\star T_QX_2$ is a fixed linear subspace
of $\PP^n$ containing $P\star Q$, $V_1\star Q$ and $P\star T_QX_2$ span a linear subspace of dimension $\min \{n,\dim X_1+\dim X_2\}$, concluding the proof for $k=2$.

Now assume $k>2$ and that the result is true for the integer $k-1$. Fix $g_i\in PGL(n+1)$, $2\le i\le k-1$, and set $W:= g_2(X_2)\star\cdots \star X_k$. The closed set $W$ is irreducible and $W\nsubseteq \Delta_{n-1}$. By the inductive assumption we have $\dim W =\min \{n,\dim X_2+\cdots +X_k\}$.
Thus the case $k=2$ gives $\dim g_1(X_1)\star\cdots \star g_{k-1}(X_{k-1})\star X_k =\min \{n,\dim X_1+\cdots +\dim X_k\}$.
\end{proof}

\begin{proof}[Proof of Theorem \ref{i3}:]
We have $g_k(X_k)\nsubseteq \Delta _{n-1}$ for a general $g_k$ and hence Theorem \ref{i3} follows from Theorem \ref{i2}.
\end{proof}

\begin{remark}\label{rel0}
Take $X_1,\dots ,X_k$ as in Theorem \ref{i2} and \ref{i3}. Assume that each inclusion $X_i\subset \PP^n$ is defined over $\RR$, i.e. the homogeneous ideal $\II[X_i]$ of $X_i$ is generated by homogeneous polynomials with real coefficients. The real algebraic group $SL(n+1,\RR)$ and its image in the complex projective group $PGL(n+1)$ is Zariski dense.
Thus in the set-up of Theorem \ref{i2} and Theorem \ref{i3}) we may take all $g_h$ with real coefficient. However, the real locus $(X_1\star \cdots \star X_k)(\RR)$ may be larger that the Hadamard product of the real loci $X_i(\RR)\subset \PP^n(\RR)$. However, if each $X_i$ had a smooth point, then each $X_i(\RR)$ contains a real differential manifold of dimension $\dim X_i$ and the Hadamard product of all sets $g_i(X(\RR))$ contains a differential manifold of dimension $\min \{n,\dim X_1+\cdots +\dim X_k\}$.
\end{remark}
\begin{proof}[Proof of Theorem \ref{i4}:]
Assume $W =X\star Y$ for some curves $X, Y$. If $X\subset H_i$ for some $i$, then $X\star Y\subseteq H_i$ for all $Y$. Thus we may assume
$X\cap \Delta_2\ne \emptyset$ and $Y\cap \Delta_2\ne \emptyset$. Set $X_0:= X\setminus X\cap \Delta_2$ and $Y_0:= Y\cap \Delta_2$.

Since $W$ is general, we may assume $W\cap \{[1:0:0:0],[0:1:0:0:0],[0:0:1:0],[0:0:0:1]\}=\emptyset$, that each $W\cap H_i$, $0\le i\le 4$, is a smooth plane curve and that each set $W\cap H_i\cap H_j$, $0\le i<j\le 3$, is formed by $d$ distinct points.

Since $d\ge 4$ and $W$ is general, the Noether-Lefschetz theorem (\cite[\S 5.3.3, Example at p. 150]{Voi07}) gives that any curve contained in $W$ is the complete intersection of $W$ with another surface and in particular its degree is divisible by $d$. The family $P\star Y$, $P\in X_0$ , (resp. $X\star Q$, $Q\in Y_0$) is an irreducible $1$-dimensional family of curves contained in $W$ and projectively equivalent to $Y$ (resp. $X$).
Thus there are integers $d_1$ and $d_2$ such that $\deg (X)=d_1d$ and $\deg(Y) =d_2d$, $X$ is projectively equivalent to the complete intersection of $W$ and a degree $d_1$ surface
and $Y$ is projectively equivalent to the complete intersection of $W$ and a degree $d_2$ surface.

Fix $i\in \{0,1,2,3\}$. Since $H_i$ is a hyperplane, $X\cap H_i\ne \emptyset$ and $Y\cap H_i\ne \emptyset$. Fix $Q_i\in H_i\cap Y$. Note that $X_0\star Q$ is a well-defined quasi-projective curve projectively equivalent to $X_0$. By definition $X\star Q$ is the closure of $X_0\star Q$ in $\PP^3$. Since $Q\in H_i$, $X\star Q\subset H_i$.
Thus $X\star Q$ is the smooth curve $W\cap H_i$ and in particular $d_1=1$, $X$ is smooth and $X$ is projectively equivalent to $W\cap H_i$. Taking $P\in X\cap H_i$
we get $d_2=1$, $P\star Y=W\cap H_i$ and that $X$ and $Y$ are projectively equivalent. 
First assume $\ge 5$. Consider the following parameter space $\Uu_d$, An element of $\Uu_d$ is a quadruple $(M_X,M_Y,X,Y)$, where $M_X$ and $M_Y$ are plane curves $X\subset M_X$ is an irreducible degree $d$ plane curve such that $X\nsubseteq \Delta_2$ and $Y\subset M_Y$ is a degree $d$ plane curve such that $Y\nsubseteq \Delta_2$. The map
$\phi: \Uu_d\to |\Oo_{\PP^3}(d)|$ is the map $(M_X,M_Y,X,Y)\mapsto X\star Y$. All possible $W$'s must be elements of $\phi(\Uu_d)$. The set of all planes of $\PP^3$ is an irreducible variety of dimension $3$. The set of all degree $d$ plane curves is a projective space of dimension $\binom{d+2}{2}-2$. Thus $\dim \Uu_d =3+3+(d+2)(d+1) -2< \binom{d+3}{3}-1$. 

Now assume $d=4$. Instead of $\Uu_4$ we take as a parameter space the variety $\Uu \subset \Uu_4$ formed by all quadruples $(M_X,M_Y,X,Y)$ with $Y$ projectively equivalent to $X$. Since $\dim PGL(3) =8$, we have $\dim \Uu = 3+3+14+8< 35-1$.
\end{proof}
\begin{remark}\label{p1}
Let $V\subsetneq \PP^n$, $V\ne \emptyset$, be an irreducible variety such that $V\nsubseteq \Delta_{n-1}$. Let $\II[V]\subseteq  \KK[x_0,\dots ,x_k]$ denote the ideal generated by all homogeneous polynomial vanishing on $V$. Since $V$ is irreducible, the ideal $\II[V]$ is contained in the inessential ideal $(x_0,\dots ,x_n)$ and it  is saturated. The homogeneous ideal is prime (\cite[Th. 8.5]{CLO}). Since $V\nsubseteq \Delta_{n-1}$ and $V$ is irreducible, $V\nsubseteq H_i$ for any $i$. Since $\II[V]$ is a prime ideal, $\II[V]:(x_i) =\II[V]$ for all $i$. Hence the assumption $\II[V]:(x_0\cdots x_n)=\II[V]$
of \cite[Th. 5.3]{ABCDHRVWY24} is satisfied by the variety $V$.
\end{remark}

\begin{proof}[Proof of Theorem \ref{i1}:]
Since $\dim Y=1$, we have $\dim X\star Y\le \min \{n,\dim X+1\}$ (\cite[p. 4]{BC24}).

Fix general $(P,Q)\in Y\times Y$. Hence $P\notin \Delta_{n-1}$, $Q\notin \Delta_{n-1}$ and $P\ne Q$. Obviously, $X\times P$ and $X\times Q$ are projectively equivalent to $X$ and in particular they have dimension $\dim X$. Thus we may assume $X\ne \PP^n$. With this assumption it is sufficient to prove that $\dim X\star Y >\dim X$.
 Since $X\star Y$ is irreducible and $\dim X\star Y\le \dim X+1$, it is sufficient to prove that $X\star P\ne X\star Q$. Fix $P$. Since $Q\notin \Delta_{n-1}$, the point $P^{-1}\in \PP^n\setminus \Delta_{n-1}$ is defined and use that $X\star P\ne X\star Q$ for a general $Q\in Y$, because $X$ is not contained in a binomial hypersurface and we may apply Remark \ref{p1} (\cite[Th. 5.3]{ABCDHRVWY24},
 \cite[Th. 7.4]{BC24}).
\end{proof}

\begin{proof}[Proof of Corollary \ref{i01}:]
It is sufficient to observe that $X\nsubseteq \Delta_{n-1}$ for any secant variety $X$ of $Y$.
\end{proof}

The author declares no conflict of interest.

\end{document}